\documentclass[12pt, reqno]{amsart}
 \usepackage{amsmath, amsthm, amscd, amsfonts, amssymb, graphicx, color, float}
\usepackage[bookmarksnumbered, colorlinks, plainpages]{hyperref}
\usepackage{tensor}
\input{mathrsfs.sty}
\hypersetup{colorlinks=true,linkcolor=red, anchorcolor=green, citecolor=cyan, urlcolor=red, filecolor=magenta, pdftoolbar=true}

\usepackage{graphicx}
\usepackage{tikz}
\usepackage{mathrsfs}
\usepackage{tkz-euclide}
\usetikzlibrary{decorations.text,calc,arrows.meta}
\usepackage{tkz-berge}
\newtheorem{theorem}{Theorem}[section]
\newtheorem{lemma}[theorem]{Lemma}
\newtheorem{proposition}[theorem]{Proposition}
\newtheorem{corollary}[theorem]{Corollary}
\theoremstyle{definition}
\newtheorem{definition}[theorem]{Definition}

\theoremstyle{remark}

\numberwithin{equation}{section}

\begin{document}

\title [On symmetric and approximately symmetric operators] {On symmetric and approximately symmetric operators}
\author[D. Khurana ]{ Divya Khurana }

\address[Khurana]{IIM Ranchi, Suchana Bhawan, Audrey House Campus, Meur’s Road, Ranchi, Jharkhand-834008, India}
\email{divyakhurana11@gmail.com, divya.khurana@iimranchi.ac.in}

\renewcommand{\subjclassname}{\textup{2020} Mathematics Subject Classification}
\begin{abstract}
We introduce the notion of local orthogonality preserving operators to study the right-symmetry of operators.  As a consequence of our work, we show that any smooth compact operator defined on a smooth and reflexive Banach space is either a rank one operator or it is not right-symmetric. We show that there are no right-symmetric smooth compact operators defined on a smooth and reflexive Banach space that fails to have any non-zero left-symmetric point. We also study approximately orthogonality preserving and reversing operators (in the sense of Chmieli\'{n}ski and Dragomir). We show that on a finite-dimensional Banach space, an operator is approximately orthogonality preserving (reversing) in the sense of Dragomir if and only if it is an injective operator.
\end{abstract}
\maketitle
\section{Introduction}
Left-symmetric and right-symmetric operators have been studied by many researchers (for more details see \cite{PMW}, \cite{GPSBOOK}, \cite{SGP}, \cite{T1}, \cite{T2}, and the references cited therein). In this article, we study right-symmetric operators in the sense of Birkhoff-James and approximately orthogonality preserving/reversing operators in the sense of  Chmieli\'{n}ski and Dragomir.

We first provide the necessary notations and basic background for our work. Throughout the text symbols $X, Y$ are used for real normed linear spaces. By $B_X=\{x:x\in X, \|x\|\leq1\}$ we denote the closed unit ball of $X$ and by $S_X=\{x:x\in X, \|x\|=1\}$ we denote the unit sphere of $X$. By $X^*$ we denote the dual space of  $X$. We use the symbol $\mathcal{B}(X,Y)$ ($\mathcal{K}(X,Y)$) to denote the space of all bounded (compact) linear operators defined from $X$ to $Y$. For the case $X=Y$, we use symbol $\mathcal{B}(X)$ ($\mathcal{K}(X)$) to denote the space of all bounded (compact) linear operators defined on $X$. If $T\in \mathcal{B}(X,Y)$  and $x\in S_X$ such that $\|Tx\|=\|T\|$ then $x$ is said to be a norm attainment element for $T$. We use $M_T=\{x\in S_X:\|Tx\|=\|T\|\}$ to denote the collection of all norm attainment elements for $T$. For $T\in \mathcal{B}(X,Y)$ we use the symbol $[T]$ to denote $[T]=\inf \{\|Tx\|:x\in S_X\}$.

For $0\not=x\in X$, we denote $\mathcal{J}(x)=\{f: f\in S_{X^*}, f(x)=\|x\|\}$. It follows from a simple application of the Hahn-Banach theorem that  $\mathcal{J}(x)$ is always a non-empty for each $0\not=x\in X$. Also, observe that $\mathcal{J}(x)$ is a convex set for each $0\not=x\in X$. We say that $0\not=x\in X$ is a smooth point if $\mathcal{J}(x)$ contains only one element.  If $\mathcal{J}(x)$ contains only one element for each $0\not=x\in X$, then we say that $X$ is a smooth space. If for any $x,y\in S_{X}$, $\|tx+(1-t)y\|=1$ for some $0\leq t\leq 1$ implies $x=y$, then $X$ is said to be a strictly convex space.

For normed linear spaces, Birkhoff-James orthogonality (\cite{Birkhoff}, \cite {James2}) is one of the natural generalizations of the usual inner product orthogonality. For $x,y\in X$, we say that $x$ Birkhoff-James orthogonal to $y$ if $\|x+\lambda y\|\geq \|x\|$ for all scalars $\lambda\in\mathbb{R}$. By using the symbol $x\perp_By$  we denote $x$ is Birkhoff-James orthogonal to $y$. In \cite[Theorem 2.1]{James2}, the authors proved that for $0\not=x\in X$,
$x\perp_B y\mbox{~if~ and~ only ~if} f(y)=0$ for some $f\in \mathcal{J}(x)$.
 By the homogeneity property of Birkhoff-James orthogonality, we can restrict ourselves to norm one element. If $\mathcal{A}\subseteq X$ and $x\in X$, we write $\mathcal{A}\perp_B x$ ($x\perp_B \mathcal{A}$) if $a\perp_B x$ ($x\perp_B a$) for all $a\in\mathcal{A}$. 
 
There are various notions of approximate Birkhoff-James orthogonality. In this work, we consider two of these notions. In \cite{Drag}, the following notion of approximate  Birkhoff-James orthogonality was introduced by  Dragomir.  For $x,y\in X$ and $\varepsilon\in[0,1)$, we say that $x$
is approximately Birkhoff-James orthogonal to $y$ if
$\|x+\lambda y\|\geq (1-\varepsilon) \|x\|$ for all
scalars $\lambda\in\mathbb{R}$. We use the following modified version of this approximate Birkhoff-James orthogonality introduced by Chmieli\'{n}ski for our work.  
For $x,y\in X$ and $\varepsilon\in[0,1)$, we say that $x$ is 
approximately Birkhoff-James orthogonal to $y$ if $\|x+\lambda y\|\geq
\sqrt{1-\varepsilon^2}\|x\|$ for all scalars $\lambda\in\mathbb{R}$ (see \cite{C} for details). By using the symbol  $x\perp_D^\varepsilon y$ we denote $x$ is approximately Birkhoff-James orthogonal to $y$ in this notion. In \cite[Lemma 3.2]{MSP}, the authors proved that for $0\not=x\in X$, $y\in X$ and $\varepsilon\in[0,1)$, $x\perp_D^\varepsilon y$ if and only if there exists $f\in S_{X^*}$ such that
\begin{align}\label{Dchar}
\ |f(x)|\geq \sqrt{1-\varepsilon^2}\|x\|
\mbox{~ and~} f(y)=0.
\end{align}

In \cite{C}, Chmieli\'{n}ski defined another variation of approximate
Birkhoff-James orthogonality. For $x,y \in X$ and $\varepsilon\in
[0,1)$, we say that  $x$ is approximately orthogonal to $y$ if $\|x+\lambda y\|^2\geq
\|x\|^2-2\varepsilon\|x\|\|\lambda y\|$ for all scalars $\lambda \in
\mathbb{R}$. We use the notation 
$x\perp_B^\varepsilon y$ to denote this variation of approximate
Birkhoff-James orthogonality. In \cite {CSW}, the authors proved that for $0\not=x\in X$, $y\in X$ and $\varepsilon\in[0,1)$,
$x\perp_B^\varepsilon y \mbox{~if~and~only~if~}|f(y)|\leq \varepsilon \|y\| \mbox{~for~some~}f\in \mathcal{J}(x)$.

The following characterization of smooth points in terms of the right-additivity of Birkhoff-James orthogonality relation obtained in  \cite{James2} will be used in our work.  $0\not=x\in X$ is a smooth point in $X$ if and only if 
$$
x\perp_B y\ \ \mbox{and}\ \ x\perp_B z\mbox{~implies~that~} x\perp_B y+z\mbox{~for~all~}y,z\in X.
$$

On a normed linear space, the Birkhoff-James orthogonality relation is not always symmetric. For this reason, the following notations of the local symmetry of Birkhoff-James orthogonality were introduced in \cite{S}.

\begin{definition}
Let ${X}$ be a normed linear space and $x\in X$. Then $x$ is left-symmetric if $x\perp_B y$ implies $y\perp_B x$ for all $y\in {X}$. 

Similarly, $x$ is said to be right-symmetric if $y\perp_B x$ implies $x\perp_B y$ for all $y\in Y$.
\end{definition}

 In \cite{CW}, the authors gave the following definition.  Let $\varepsilon\in[0,1)$. Then Birkhoff-James orthogonality is said to be $\varepsilon$-symmetric in a normed linear space $X$ if $x\perp_B y$ implies $y\perp_B^\varepsilon x$ for all $x,y\in X$.

We say that $T\in\mathcal{B}(X)$ is orthogonality preserving if $x\perp_B y$ implies $Tx\perp_B Ty$ for all $x,y\in X$. In \cite{CAOT}, the author defined the following notion of orthogonality reversing operator.
 $T\in\mathcal{B}(X)$ is said to be orthogonality reversing if $x\perp_B y$ implies $Ty\perp_B Tx$ for all $x,y\in X$. 

 In \cite{T3} and  \cite{T4}, the authors proved that if $T\in\mathcal{B}(X)$ is orthogonality preserving then $T$ is a scalar multiple of a linear isometry, where $X$ is a Banach space. It was proved in \cite{WOR} that if $dim~X\geq 3$ then $X$ has an orthogonality reversing operator if and only if $X$ is an inner product space. For more details on orthogonality preserving and reversing operators, we refer the readers to  \cite{T3}, \cite{CAOT}, \cite{GPSBOOKC}, \cite{T4}, \cite{WOR}, and the references cited therein.

Motivated by these definitions we now introduce the following local version of one-sided orthogonality preserving and reversing operators in the sense of Birkhoff-James. We also introduce the notion of local and global versions of approximate orthogonality preserving and reversing operators in the sense of Chmieli\'{n}ski and Dragomir.

\begin{definition}
    Let $X,Y$ be normed linear spaces and $T\in\mathcal{B}(X,Y)$. Let $x\in S_X$. We say that $T$ is left orthogonality preserving (reversing) at $x$ if $x\perp_B y$ implies $Tx\perp_BTy$ ($Ty\perp_BTx$) for all $y\in X$.

    In a similar way we can define right orthogonality preserving and reversing operators at $x$.
\end{definition}

 We now give a few examples of operators that are locally left-orthogonality preserving. Let $X$ be a smooth Banach space. Let $T\in\mathcal{B}(X)$ and $x\in M_T$. Then \cite[Theorem 2.4]{SJMAA} implies that $T$ is left orthogonality preserving at $x$.

Let $X,Y$ be normed linear spaces. Let $x\in S_X$ be a smooth point and $y\in S_X$. Let $H_x$ be a hyperplane in $X$ such that $x\perp_BH_x$. Let $\{x_\alpha:\alpha\in\Lambda\}$ be a Hamel basis for $H_x$. Then $\{x,x_\alpha: \alpha\in\Lambda\}$ is a Hamel basis for $X$. We now define $T:X\longrightarrow Y$ by $Tx=y$ and $Tx_\alpha=0$ for all $\alpha\in\Lambda$. Then $T\in \mathcal{B}(X,Y)$ and $x\in M_T$. It follows from the right additivity of Birkhoff-James orthogonality at $x$ that any $z\in X$ with $x\perp_Bz$ has to be an element of $H_x$. Thus, it follows $T$ that $T$ is left orthogonality preserving at $x$.

Now, consider $X=\ell_\infty^2$ and $x=(1,0)\in X$. Let $f_x=(1,0)\in \ell_1^2$. Then $\mathcal{J}(x)=\{f_x\}$ and $y\in S_X$ with $x\perp_By$ has only two possible options $(0,1)$ and $(0,-1)$ (see Figure 1). Let $y=(0,1)\in X$.

\begin{center}
\begin{tikzpicture}[scale=1.1]
\draw (-3.8,1.2)  node {$z$};
\draw [fill] (-4,1) circle [radius=.05];
\draw[thick] (-6,1)--(-4,1) node[right]{};
\draw[thick] (-6,-1)--(-4,-1) node[right]{};
\draw[thick] (-4,1)--(-4,-1) node[right]{};
\draw[thick] (-6,1)--(-6,-1) node[right]{};
\draw[thick, dashed] (-7.2,0)--(-2.8,0) node[above]{};
\draw[thick, red, dashed] (-5,-2)--(-5,2) node[above]{};
\draw (-3.8,0.2)  node {$x$};
\draw (-5.2,1.2)  node {$y$};
\draw (-5.4,-1.2)  node {$-y$};
\draw [fill] (-4,0) circle [radius=.05];
\draw [fill] (-5,1) circle [radius=.05];
\draw [fill] (-5,-1) circle [radius=.05];
\draw[thick, red, dashed] (-4,2)--(-4,-2) node[right]{};
\draw (-4.2,-1.6) node {$f_x$};
\draw (-5.6,-1.6) node {$\mbox{ker} ~f_x$};
\draw (-5,-2.5) node {Figure 1};
\end{tikzpicture}
\end{center}

Clearly, $X=\mbox{span}\{x,y\}$. We define $A:X\longrightarrow X$ by $A(\alpha x+\beta y)=\alpha (1,1)$ for $\alpha,\beta\in \mathbb{R}$. By the definition of $A$ and our choice of $x,y$ it follows that $x\in M_A$ and $T$ is left orthogonality preserving at $x$. Observe that, if we take $z=(1,1)\in X$ then $z\perp_B x$ but $Tz\not\perp_B Tx$, so $T$ is not orthogonality preserving.

\begin{definition}
    Let $X, Y$ be normed linear spaces. Let $\mathcal{A}\subseteq X$ and $\varepsilon_\mathcal{A}\in[0,1)$. We say that $T\in\mathcal{B}(X,Y)$ is  left $\varepsilon_\mathcal{A}$-approximately orthogonality preserving (reversing) with respect to $\mathcal{A}$ in the sense of Chmieli\'{n}ski, in short, $\varepsilon_\mathcal{A}-CLAOP$ ($\varepsilon_\mathcal{A}-CLAOR$), if $x\perp_B y$ implies $Tx\perp_B^{\varepsilon_\mathcal{A}} Ty$ ($Ty\perp_B^{\varepsilon_\mathcal{A}} Tx$) for all
$x\in\mathcal{A}, y\in X$. 

In this case if $\mathcal{A}=\{x\}$ for some $x\in X$ we simply write  $T$ is $\varepsilon_x-CLAOP$ ($\varepsilon_x-CLAOR$) at $x$. 

In a similar way we can define operators which are right $\varepsilon_\mathcal{A}$-approximately orthogonality preserving (reversing) with respect to $\mathcal{A}$ in the sense of Chmieli\'{n}ski
\end{definition}
\begin{definition}
    Let $X,Y$ be normed linear spaces and $T\in\mathcal{B}(X,Y)$. Let $x\in S_X$ and $\varepsilon\in[0,1)$. We say that $T$ is $\varepsilon$-approximately orthogonality preserving (reversing) in the sense of Dragomir, in short, $\varepsilon-DAOP$ ($\varepsilon-DAOR$), if $x\perp_B y$ implies $Tx\perp_D^\varepsilon Ty$ ($Ty\perp_D^\varepsilon Tx$) for all $x,y\in X$.
\end{definition}

The following result was obtained in \cite{SPMR}. An operator $T\in S_{\mathcal{K}(X,Y)}$ is smooth if and only if $M_T=\{\pm x_0\}$ and $Tx_0$ is a smooth point in $Y$, where $X$ is a reflexive Banach space, $Y$ is a normed linear space and $x_0\in S_X$. 

We now recall some of the results related to the right-symmetry of operators. In \cite[Theorem 2.3]{SGP}, the authors proved that any smooth compact operator on a finite-dimensional smooth and strictly convex Banach space can not be right-symmetric. In \cite [Theorem 15.6]{GPSBOOK}  
   the authors proved a similar result on a smooth, reflexive, and strictly convex Banach space. In \cite[Theorem 3.1]{PMW}, the authors proved that any smooth compact operator from a strictly convex and reflexive Banach space to a strictly convex normed linear space can not be right-symmetric. \\
   
   In section 2, we continue a similar study related to the above-mentioned results on right-symmetric operators.  We show that if a right-symmetric compact operator $T\in S_{\mathcal{K}(X,Y)}$ with $M_T=\{\pm x_0\}$ is left orthogonality preserving at $x_0$ then $x_0$ is left-symmetric, where $X$ is a reflexive Banach space, $Y$ is a normed linear space and $x_0\in S_X$. This result has the following direct consequence. If $X$ is a smooth and reflexive Banach space such that $S_X$ fails to have any left-symmetric point then any smooth compact operator on $X$ is not right-symmetric. We show that if $T\in S_{\mathcal{K}(X,Y)}$ is a right-symmetric compact operator such that $M_T=\{\pm x_0\}$ and $T$ is left orthogonality preserving at $x_0$ then either $T$ is a rank one operator or $T$ is not right-symmetric, where $X$ is a reflexive Banach space, $Y$ is a smooth normed linear space and $x_0\in S_X$. As a corollary, we show that any smooth compact operator defined on a smooth and reflexive Banach space is either a rank one operator or it is not right-symmetric. We also show that for a rank one smooth compact operator $T$ defined on a smooth and reflexive Banach space with dimension at least 2, it is necessary that $Tx_0\in Ker~T$, where $M_T=\{\pm x_0\}$ for some $x_0\in S_X$.

In section 3, we study $\varepsilon$-DAOP, $\varepsilon$-DAOR and  $\varepsilon_\mathcal{A}$-CLAOR operators. We show that on a finite-dimensional Banach space an operator is $\varepsilon$-DAOP (or $\varepsilon$-DAOR) if and only if it is injective.

\section{Right-Symmetric operators}\label{section 2}
We start this section with the following simple observation.
\begin{lemma}\label{lem1}
    Let $X$ and $Y$ be normed linear spaces. Let $T\in S_{\mathcal{B}_{(X,Y)}}$ and $x_0\in S_X$. If $x_0\in M_T$ then $x_0\perp_B Ker ~T$.
\end{lemma}
\begin{proof}
    Let $y\in Ker~ T$. Then $Ty=0$ and
    $$1=\|Tx_0\|=\|Tx_0+\lambda Ty\|\leq \|T\|\|x_0+\lambda y\|$$
    for all $\lambda\in\mathbb{R}$. Thus, $\|x_0+\lambda y\|\geq \|x_0\|$ for all $\lambda\in\mathbb{R}$ and the result follows.
\end{proof}

 In the next result, we show that for any compact operator $T\in S_{\mathcal{K}{(X, Y)}}$, if $Ker~T\not\perp_B x_0$ then $T$ can not be right-symmetric, where $M_T=\{\pm x_0\}$ for some $x_0\in S_X$, $X$ is a reflexive Banach space and $Y$ is a normed linear space. 

\begin{theorem}
    Let $X$ be a reflexive Banach space. Let $Y$ be a normed linear space and $T\in S_{\mathcal{K}{(X,Y)}}$  such that $M_T=\{\pm x_0\}$ for some $x_0\in S_X$. Then, either of the following is true:
    \begin{itemize}
        \item [(i)] $Ker~T\perp_B x_0.$
        \item[(ii)] $T$ is not right-symmetric.
    \end{itemize}
    \end{theorem}

\begin{proof}
    If $Ker~T\perp_B x_0$ then (i) follows. 
    
    Let $Ker~T\not\perp_B x_0$ and $T$ is right-symmetric. Then there exists $y\in S_X\cap Ker~T$ with $y\not\perp_B x_0$. 
    Let $H_{y}$ be a hyperplane such that $y\perp_B H_{y}$. Then it follows from our choice of $y$ that $x_0\not\in H_{y}$. We now define the following map $$A:X\longrightarrow Y$$ by $A(\alpha y+h)=\alpha Tx_0$ for $\alpha\in \mathbb{R}$ and $h\in H_{y}$. Clearly, $A$ is a compact operator and $y\in M_A$. Also, $y\in Ker~T$ implies that $Ay\perp_B T y$. Since $y\in M_A$, we get $A\perp_B T$.
    
    But $x_0\not\in H_y$ implies that there exist $0\not=a_{x_0}\in \mathbb{R}$ and $0\not=h_{x_0}\in H_y$ such that $x_0={a_{x_0}}y+h_{x_0}$. Thus,  $Tx_0\not\perp_B Ax_0$. Now, by using \cite[Theorem 2.1]{PSGD} we get $T\not\perp_B A$, a contradiction. Thus (ii) follows.
    \end{proof}    

We now obtain a characterization of a Hilbert space of dimension at least 3 in terms of the orthogonality relationship of $x_0\in M_T$ and $Ker~T$ for any $T\in S_{\mathcal{K}_{(X)}}$, where $X$ is a strictly convex reflexive  Banach space with dimension atleat 3.
\begin{theorem}
    Let $X$ be a strictly convex reflexive  Banach space with $dim~X\geq 3$. Then, the following are equivalent.
    \begin{itemize}
        \item [(i)] $X$ is a Hilbert space.
        \item[(ii)]  $Ker~T\perp_B x_0$ for each $T\in S_{\mathcal{K}(X)}$ with $M_T=\{\pm x_0\}$ for some $x_0\in S_X$. 
    \end{itemize}
\end{theorem}
\begin{proof}
    (i) $\implies$ (ii)  Birkhoff-James orthogonality is symmetric in a  Hilbert space. Thus, the result follows by using Lemma~\ref{lem1}.
    
    We now prove (ii) $\implies$ (i). Let $x_0\in S_X$ and $H_{x_0}$ be a hyperplane such that $x_0\perp_B H_{x_0}$. We now define the following map $$T:X\longrightarrow X$$ by $T(\alpha x_0+h)=\alpha x_0$ for $\alpha\in \mathbb{R}$ and $h\in H_{x_0}$. Clearly, $T$ is a compact operator. Also, by using strict convexity of $X$, it follows that $M_T=\{\pm x_0\}$. It follows from the definition of $T$ that $Ker~T=H_{x_0}$. Now, (ii) implies that $Ker~T\perp_B x_0$ and thus $H_{x_0}\perp_B x_0$. Thus, (ii) $\implies$ (i) follows by using \cite[Theorem 4.14]{SUR}.
\end{proof}

In the following result, we establish a relationship between the smoothness of $x_0$ and $Tx_0$ for any operator $T\in S_{\mathcal{B}(X,Y)}$ which is left orthogonality preserving at $x_0$, for some $x_0\in M_T$, where $X,Y$ are normed linear spaces and $x_0\in S_X$.
\begin{lemma}\label{lemmasmoothness}
    Let $X$ and $Y$ be normed linear spaces. Let $T\in S_{\mathcal{B}(X,Y)}$ be such that $x_0\in M_T$ for some $x_0\in S_X$. Suppose that $T$ is left orthogonality preserving at $x_0$.
    \begin{itemize}
        \item[(i)] If $Tx_0$ is smooth then $x_0$ is smooth.
        \item[(ii)] If $x_0$ is smooth and $T$ is onto then $Tx_0$ is smooth.
    \end{itemize}
\end{lemma}
\begin{proof}
    (i) Let $y,z\in S_X$ be such that $x_0\perp_B y$ and $x_0\perp_B z$. Then by using $T$ is left orthogonality preserving at $x_0$, we get $Tx_0\perp_B Ty$ and $Tx_0\perp_B Tz$. Now, by using the smoothness of $Tx_0$, we get $Tx_0\perp_B Ty+Tz$. Thus, $x_0\in M_T$ implies that $x_0\perp_B y+z$ and the result follows.

    (ii) Let $y,z\in Y$ be such that $Tx_0\perp_B y$ and $Tx_0\perp_B z$. Let $x_1,x_2\in X$ be such that $Tx_1=y$ and $Tx_2=z$. Then $x_0\in M_T$ implies that $x_0\perp_B x_1$ and $x_0\perp_B x_2$. Now, by using the smoothness of $x_0$, we get $x_0\perp_B x_1+x_2$. Thus, by using the left orthogonality preserving property of $T$ at $x_0$, we get $Tx_0\perp_B y+z$, and the result follows.
\end{proof}

We now study a connection between the left-symmetry of $x_0\in M_T$ and the right-symmetry of an operator $T\in S_{\mathcal{K}(X,Y)}$ which is left orthogonality preserving at $x_0$, where  $M_T=\{\pm x_0\}$ for some $x_0\in S_X$, $X$ is a reflexive Banach space and $Y$ is a normed linear space. 
\begin{proposition}\label{lem2}
    Let $X$ be a reflexive Banach space. Let $Y$ be a normed linear space and $T\in S_{\mathcal{K}({X,Y})}$ such that $M_T=\{\pm x_0\}$ for some $x_0\in {S_{{X}}}$. Let $T$ be left orthogonality preserving at $x_0$. If $T$ is right-symmetric then $x_0$ is left-symmetric.
\end{proposition}
     
\begin{proof}
   Let $x_0$ is not left-symmetric. Then we can find $y\in S_{X}$ be such that $x_0\perp_B y$ and $y\not\perp_B x_0$. Since $T$ is left orthogonality preserving at $x_0$, we get $Tx_0\perp_B Ty$. 
   
   Let $H_y$ be a hyperplane such that $y\perp_B H_y$. Then $x_0\not\in H_y$. Define a linear operator $A:{X}\longrightarrow {Y}$ by
   $$A(\alpha y+h)=\alpha Tx_0$$
   for all $\alpha\in\mathbb{R}$ and $h\in H_y$. 
   
   Clearly, $A$ is a compact operator and it follows from our choice of $H_y$ that $y\in M_A$. Now, by using  $Tx_0\perp_B Ty$ we get $Ay\perp_B Ty$. Since $y\in M_A$, thus $A\perp_B T$. Now, from the right-symmetry of $T$, we have $T\perp_B A$. Thus, it follows from \cite[Theorem 2.1]{PSGD} that $Tx_0\perp_B Ax_0$. But the definition of $A$ and $x_0\not\in H_y$ implies that $Tx_0\not\perp_B Ax_0$, a contradiction.
\end{proof}

It is known that  $\ell_1^n$, $n\geq 3$, and $\ell_1$ fail to have any non-zero left-symmetric point (for details see \cite{BRS},\cite{D}). Also, in \cite{D}, many other examples of Banach spaces that fail to have any non-zero left-symmetric point are constructed. For such spaces, we have the following immediate consequence of Proposition~\ref{lem2}

\begin{corollary}
     Let $X$ be a reflexive Banach space such that $S_X$ has no left-symmetric points. Let $Y$ be a normed linear and $T\in S_{\mathcal{K}({X,Y})}$ such that $M_T=\{\pm x_0\}$ for some $x_0\in {S_{{X}}}$. If $T$ is left orthogonality preserving at $x_0$ then $T$ is not right-symmetric.
\end{corollary}

The following result follows by using \cite[Theorem 2.4]{SJMAA} and Proposition~\ref{lem2}.

\begin{corollary}\label{corsmooth}
     Let $X$ be a smooth and reflexive Banach space. Let $T\in S_{\mathcal{K}({X})}$ be such that $M_T=\{\pm x_0\}$ for some $x_0\in {S_{{X}}}$. If $T$ is right-symmetric, then $x_0$ is left-symmetric.
\end{corollary}

 We now study the right-symmetry of a smooth compact operator $T$, which preserves left orthogonality at the norm attainment element of $T$.

\begin{theorem}\label{thmrank1}
   Let $X$ be a reflexive Banach space. Let $Y$ be a smooth normed linear space and $T\in S_{\mathcal{K}({X,Y})}$ such that $M_T=\{\pm x_0\}$ for some $x_0\in {S_X}$.  If $T$ is left orthogonality preserving at $x_0$ then either of the following is true.
   \begin{itemize}
       \item [(i)] $dim~T(X)=1$.
       \item[(ii)] $T$ is not right-symmetric.
   \end{itemize}
  
\end{theorem}
\begin{proof}
   If $dim~T(X)=1$, then (i) follows. 
   
   We now assume that $T\in S_{\mathcal{K}({X,Y})}$ is right-symmetric with $M_T=\{\pm x_0\}$ for some $x_0\in {S_{{X}}}$ and $dim~T(X)>1$.  

    We first claim that there exists an element $y\in S_X$ with $x_0\perp_B y$, and $Tx_0, Ty$ are linearly independent. Let $H$ be a hyperplane such that $x_0\perp_B H$ and $Ty\in span\{Tx_0\}$ for all $y\in H$. Then $T(ax_0+y)=aT(x_0)+T(y)\in span\{Tx_0\}$ for all $y\in H$ and $a\in \mathbb{R}$. This shows that $dim~T(X)=1$, a contradiction. Thus, our claim follows. 

    Let $y_0\in S_X$ be such that $x_0\perp_B y_0$ and $Tx_0, Ty_0$ are linearly independent.
   
   It follows from Lemma~\ref{lemmasmoothness} that $x_0$ is smooth. Now, smoothness of $x_0$ implies that $x_0\not\perp_B x_0+\frac{y_0}{2}$. Let $z=\frac{x_0+\frac{y_0}{2}}{\|x_0+\frac{y_0}{2}\|}$ then $x_0\not\perp_B z$. Also, from our choice of $z$, it follows that $Tz\not=0$.

    Now, $x_0\in M_T$ and $x_0\not\perp_B z$ implies that $Tx_0\not\perp_B Tz$. Let $0\not=a\in\mathbb{R}$ be such that $aTz+Tx_0\perp_B Tz$. We now show that $a\not=-{\|x_0+\frac{y_0}{2}\|}$. If $a=-{\|x_0+\frac{y_0}{2}\|}$ then $Ty_0\perp_B Tz$. This gives $Ty_0\perp_B Tx_0+\frac{1}{2}Ty_0$. Also, from our assumption of $T$ and our choice of $y_0$, it follows that $Tx_0\perp_B Ty_0$. Now, \cite[Proposition 3.2]{PMW} implies that $Tx_0$ is left-symmetric and we get $Ty_0\perp_B Tx_0$. Thus, using the smoothness of $Y$, we get $Ty_0\perp_B Ty_0$, a contradiction. 
    
    Let $H_z$ be a hyperplane such that $ z\perp_B H_z$. Define a map $A:{X}\longrightarrow {Y}$ by
   $$A(\gamma z+h)=\gamma (a Tz+Tx_0)$$
   for all $\gamma\in\mathbb{R}$ and $h\in H_z$. 
   Then clearly $A\in \mathcal{K}(X,Y)$ and $z\in M_A$. By our choice of $a$, we have $a Tz+Tx_0\perp_B Tz$. Thus, $Az=a Tz+Tx_0$ and $z\in M_A$ implies that $A\perp_B T$.
    Using \cite[Theorem 2.1]{SPLAA} and right-symmetry of $T$ we get $Tx_0\perp_B Ax_0$.  
   
   We now claim that $x_0\not\in H_z$. If $x_0\in H_z$ then $z\perp_B x_0$. This shows that $x_0+\frac{1}{2}y_0\perp_B x_0$. Now, by using $x_0+\frac{1}{2}y_0\perp_B x_0$ and $x_0\perp_B y_0$  we get 
   
   $$\frac{1}{2}=\frac{1}{2}\|y_0\|=\|(x_0+\frac{1}{2}y_0)-x_0\|\geq \|x_0+\frac{1}{2}y_0\|\geq \|x_0\|.$$
   But this contradicts $\|x_0\|=1$. Thus, $x_0\not\perp_B H_z$.
   
   Let $0\not=a_{x_0}\in\mathbb{R}$ and $0\not=h_{x_0}\in H_z$ be such that $x_0=a_{x_0}z+h_{x_0}$. Then $Ax_0=a_{x_0} (aTz+Tx_0)$. Thus, $Tx_0\perp_B Ax_0$ implies that $Tx_0\perp_B a_{x_0} (a Tz+Tx_0)$.  By using smoothness of $T{x_0}$, $a\not=-{\|x_0+\frac{y_0}{2}\|}$, $Tx_0\perp_B Ty_0$ and $z=\frac{x_0+\frac{y_0}{2}}{\|x_0+\frac{y_0}{2}\|}$, we get $Tx_0\perp_B Tx_0$, a contradiction. This proves the result.   
   \end{proof}

The following result is an immediate consequence of \cite[Theorem 2.4]{SJMAA} and Theorem~\ref{thmrank1}.
\begin{corollary}\label{Rank1}
    Let $X$ be a smooth Banach and reflexive space.  Let $T\in S_{\mathcal{K}({X})}$ be a smooth operator.  Then, either of the following is true:
   \begin{itemize}
       \item [(i)] $dim~T(X)=1$.
       \item[(ii)] $T$ is not right-symmetric.
        \end{itemize}
\end{corollary}
   
  In the following result, we study the right-symmetry of a rank one compact operator $T\in S_{\mathcal{K}({X})}$ with $M_T=\{\pm x_0\}$ for some $x_0\in {S_X}$, where $X$ is a reflexive Banach space.

\begin{proposition}
    Let $X$ be a reflexive Banach space with $dim~X\geq 2$. Let $T\in S_{\mathcal{K}({X})}$ such that $M_T=\{\pm x_0\}$ for some $x_0\in {S_X}$ and $dim~T(X)=1$. Let $T$ be left orthogonality preserving at $x_0$ and $Tx_0$ be left-symmetric. If $Tx_0\not\in Ker~T$, then $T$ is not right-symmetric.
\end{proposition}

\begin{proof}
  Let $T$ be right-symmetric. Let $T(X)=\mbox{span} \{Tx_0\}$. Let $Tx_0\not\in Ker~T$. Clearly, $dim~X\geq 2$ implies that there exists $0\not=h\in Ker~T$. This gives $I\perp_B T$ and thus right-symmetry of $T$ implies that $T\perp_B I$. Now, it follows from \cite[Theorem 2.1]{PSGD} that $Tx_0\perp_B x_0$ and thus left-symmetry of $Tx_0$ implies that $x_0\perp_B Tx_0$. By using the left orthogonality preserving property of $T$ at $x_0$, we get $Tx_0\perp_B T(Tx_0)$, a contradiction. Thus, the result follows.
\end{proof}

   It follows from \cite[Proposition 3.2]{PMW} that if $T$ is a right-symmetric compact operator on a reflexive and smooth Banach $X$ space with $M_T=\{\pm x_0\}$, where $x_0\in S_X$ then $Tx_0$ is left-symmetric. Now, the following result follows immediately by using \cite[Theorem 2.4]{SJMAA}.
\begin{corollary}\label{corkernel}
    Let $X$ be a smooth and reflexive Banach space with $dim~X\geq 2$. Let $T\in S_{\mathcal{K}({X})}$ such that $M_T=\{\pm x_0\}$ for some $x_0\in {S_X}$ and $dim~T(X)=1$. If $Tx_0\not\in Ker~T$, then $T$ is not right-symmetric.
\end{corollary}

In \cite[Theorem 2.4]{SGP} and \cite[Theorem 3.4]{PMW} the authors studied the right-symmetry of a smooth compact operator $T\in S_{\mathcal{B}(X)}$ when $x_0$ is an eigen vector for $T$, where $M_T=\{\pm x_0\}$ for some $x_0\in S_X$. In this direction, the following result follows immediately from Corollaries~\ref{Rank1} and~\ref{corkernel}.
\begin{corollary}
    Let $X$ be a smooth and reflexive Banach space with $dim~X\geq 2$. Let $T\in S_{\mathcal{K}({X})}$ such that $M_T=\{\pm x_0\}$ for some $x_0\in {S_X}$. If $x_0$ is an eigen-vector for $T$, then $T$ is not right-symmetric.
\end{corollary}

The following results follow immediately by using  \cite[Proposition 3.2]{PMW} and Proposition~\ref{lem2}.

\begin{theorem}
     Let $X$ be a reflexive Banach space. Let $Y$ be a normed linear space and $T\in S_{\mathcal{K}(X,Y)}$ be right-symmetric such that $M_T=\{\pm x_0\}$ for some $x_0\in  S_X$. Let $T$ be left orthogonality preserving at $x_0$. If $x_0$ is an eigenvector for $T$, then $x_0$ is symmetric.
\end{theorem}\label{propeigen}

\section{Approximate orthogonality preserving and reversing operators}\label{section 3}
We begin this section by proving that any $\varepsilon$-DAOP operator $T\in S_{\mathcal{B}(X,Y)}$ is one-to-one, where $X,Y$ are normed linear spaces and $\varepsilon\in[0,1)$.
\begin{theorem}\label{injective}
    Let $X, Y$ be normed linear spaces. Let $T\in S_{\mathcal{B}(X,Y)}$ such that $T$ is $\varepsilon$-DAOP for some $\varepsilon\in[0,1)$. Then $T$ is one-to-one.
\end{theorem}

\begin{proof}
We claim that $Tx\not=0$ for all $x\in S_X$. Let $Ty=0$ for some $y\in S_X$. Since $T\not=0$, there exists $x_0\in S_X$ such that $Tx_0\not=0$. Let $X_0=\mbox{span}\{x_0,y\}$ and $S_0=S_{X_0}\setminus \{\mbox{span} \{x_0\}\cup \mbox{span}\{y\}\}$. We now consider the following cases.
    \begin{itemize}
        \item [(a)]  Let $x\in S_0$ be such that $x$ is not smooth. Then $\mathcal{J}(x)$ is not a singleton set and by using the convexity of $\mathcal{J}(x)$ we can find $f_{x}\in \mathcal{J}(x)$ such that $Ker f_{x}\cap S_0\not=\emptyset$. 
        Let $x_1\in Ker f_{x}\cap S_{0}$. Then $x\perp_B x_1$. Let $x=\alpha_1x_0+\beta_1y$ and $x_1=\alpha_2x_0+\beta_2y$, where $\alpha_1,\alpha_2,\beta_1,\beta_2\not=0$. Now, by using $T$ is $\varepsilon$-DAOP for some $\varepsilon\in[0,1)$, we get   $Tx\perp_D^{\varepsilon} Tx_1$. Thus, $\beta_1 Tx_0\perp_D^{\varepsilon}\beta_2 Tx_0$, a contradiction.
        \item [(b)]  Let each $x\in S_0$ is smooth. Let $x\in S_0$ be such that  $Ker f_{x}\cap S_0\not=\emptyset$, where $\mathcal{J}(x)=\{f_{x}\}$. In this case, we arrive at a contradiction by using arguments similar to case (a).
        \item[(c)] Let each $x\in S_0$ is smooth. Let $Kerf _{x}=\mbox{span} \{x_0\}$ for each $x\in S_0$, where $\mathcal{J}(x)=\{f_{x}\}$. Choose a sequence $\{x_n\}\subseteq X_0\setminus \{\mbox{span} \{x_0\}\cup \mbox{span} \{y\}\}$ such that $x_n\longrightarrow x_0$. Then by our assumption on $S_0$, we get $x_n\perp_B x_0$ for all $n$. This gives $x_0\perp_B x_0$, a contradiction.

        \item[(d)] Let each $x\in S_0$ is smooth. Let $Kerf _{x}=\mbox{span} \{y\}$ for each $x\in S_0$, where $\mathcal{J}(x)=\{f_{x}\}$. By using arguments similar to case (c) we arrive at a contradiction.
    \item[(e)] Let each $x\in S_0$ is smooth. Let $Ker f _{x_1}=\mbox{span} \{x_0\}$  and $Kerf _{x_2}=\mbox{span} \{y\}$ for some $x_1,x_2\in S_0$, where $\mathcal{J}(x_1)=\{f_{x_1}\}$ and $\mathcal{J}(x_2)=\{f_{x_2}\}$. Let
     $x_1=\alpha_1x_0+\beta_1y$, where $\alpha_1,\beta_1\not=0$. Then by using $T$ is $\varepsilon$-DAOP for some $\varepsilon\in[0,1)$, we get $Tx_1\perp_D^{\varepsilon} Tx_0$ . Thus, $ \beta_1Tx_0\perp_D^{\varepsilon} Tx_0$, a contradiction.

      This shows that $Tx\not=0$ for all $x\in S_X$. Thus, $T$ is one-to-one.
    \end{itemize}
\end{proof}

The following result follows using arguments similar to Theorem~\ref{injective}.

\begin{theorem}\label{injective2}
    Let $X, Y$ be normed linear spaces. Let $T\in S_{\mathcal{B}(X,Y)}$ such that $T$ is $\varepsilon$-DAOR for some $\varepsilon\in[0,1)$. Then $T$ is one-to-one.
\end{theorem}

In \cite[Theorem 2.3]{CKS2}, the authors proved that on a finite dimension Banach space identity operator is always $\varepsilon$-DAOR for some $\varepsilon\in [0,1)$. We now show that an operator $T\in S_{\mathcal{B}(X,Y)}$ is $\varepsilon$-DAOP (or $\varepsilon$-DAOR) if and only if $T$ is injective, where $X$ is a finite-dimensional Banach space and $Y$ is a normed linear space. Thus,\cite[Theorem 2.3]{CKS2} follows immediately by using the following result.

\begin{theorem}
   Lex $X$ be an $n$-dimensional Banach space. Let $Y$ be a normed linear space and $T\in S_{\mathcal{B}(X,Y)}.$ Then the following are equivalent:
   \begin{itemize}
       \item [(i)] $T$ is $\varepsilon_1$-DAOP for some $\varepsilon_1\in[0,1)$.
       \item [(ii)] $T$ is $\varepsilon_2$-DAOR for some $\varepsilon_2\in[0,1)$.
       \item [(iii)] $T$ is one-to-one.
    \end{itemize}
\end{theorem}

\begin{proof}
    (i)$\implies$ (ii) Let $x,y\in S_X$ be such that $x\perp_B y$. If both $Tx$ and $Ty$ are non-zero then it follows from (i) and Eqn.~(\ref{Dchar}) that $Tx$ and $Ty$ are linearly independent.  It follows from \cite[Proposition 2.2]{CKS2} that there exists $\varepsilon_{Tx,Ty}\in[0,1)$ such that $Ty\perp^{\varepsilon_{Tx,Ty}}_B Tx$.  Let $\varepsilon_{Tx,Ty}^*$ be infimum of all such $\varepsilon_{Tx,Ty}$. Let $$
    \varepsilon=sup_{x\in S_X}sup_{y\in x^\perp\cap S_X}\varepsilon_{Tx,Ty}^*.$$ By using compactness of $S_X$, (i) and arguments similar to \cite[Theorem 2.3]{CKS2} we can show that $\varepsilon<1$. Thus,  (i)$\implies$ (ii) follows.

    (ii) $\implies$ (iii) follows by Theorem~\ref{injective2}

     (iii) $\implies$ (i) Let $x,y\in S_X$. If $x\perp_By$ then it follows from injectivity of $T$ that  $Tx$ and $Ty$  are linearly independent. Now, (i) follows by using arguments similar to the proof of (i) $\implies$ (ii).
\end{proof}

In Section 2, we have used \cite[Theorem 2.4]{SJMAA} several times.  \cite[Theorem 2.4]{SJMAA} states that if $T\in S_{\mathcal{B}(X)}$ and $x\in M_T$ such that $x,Tx$ are smooth points then $T$ is left orthogonality preserving at $x$, that is, $x\perp_B y\implies Tx\perp_B Ty$ for all $y\in X$. We now show that if $T$ is a compact operator defined on reflexive Banach space with $[T]>0$ then $T$ is $\varepsilon_x$-CLAOR  at smooth norm attainment element $x$  of $T$ for some $\varepsilon_x\in[0,1)$. 
 In the following result the notation $y_{n_k}\overset{w}\longrightarrow y$ means the sequence $\{y_{n_k}\}$ converges to $y$ in the weak topology of $X$, where $\{y_{n_k}\}\subseteq X$ and $y\in X$.

\begin{theorem}\label{thmCAOR}
    Let $X$ be a  reflexive Banach space and $T\in S_{\mathcal{K}(X)}$. Let $[T]>0$ and $x\in M_T$ for some smooth point $x\in S_X$. Then $T$ is $\varepsilon_x$-CLAOR at $x$ and for some $\varepsilon_x\in[0,1)$.
\end{theorem}

\begin{proof}
    Let $x\in M_T$ be a smooth point for some $x\in S_X$. Let $T$ is not  $\varepsilon_x$-CLAOR for any $\varepsilon_x\in[0,1)$. Then there exist $\{y_n\}\subseteq S_X$ and $\varepsilon_n\nearrow 1$ such that $x\perp_B y_n$ but $Ty_n\not\perp_B^{\varepsilon_n}Tx$. Thus, for any $f_{Ty_n}\in J(Ty_n)$ we have $|f_{Ty_n}(Tx)|>\varepsilon_n \|Tx\|=\varepsilon_n$. For every $Ty_n$ we choose one such $f_{Ty_n}\in J(y_n)$. 
    
    Now, by using reflexivity of $X$ and $X^*$ we can obtain subsequences $\{y_{n_k}\}\subseteq\{y_n\}$ and $\{f_{Ty_{n_k}}\}\subseteq\{f_{Ty_n}\}$ such that $y_{n_k}\overset{w}\longrightarrow y$  and $f_{Ty_{n_k}}\overset{w}\longrightarrow f$ for some $y\in B_X$ and $f\in B_{X^*}$. Thus, by using compactness of $T$, we get $Ty_{n_k}\longrightarrow Ty$. Also,  from $f_{Ty_{n_k}}\overset{w}\longrightarrow f$, we get $f_{Ty_{n_k}}(Tx)\longrightarrow f(Tx)$. Thus, $|f_{Ty_{n_k}}(Tx)|>\varepsilon_{n_k} \|Tx\|=\varepsilon_{n_k}$ and $\varepsilon_n\nearrow 1$ implies that $f\circ T\in S_{X^*}$ and $|f\circ T(x)|=1=\|x\|$. Using smoothness of $x$, we get $\mathcal{J}(x)=\{f\circ T\}$. 

    Now, by using $x\perp_B y_{n_k}$ and $\mathcal{J}(x)=\{f\circ T\}$, we get $f(Ty_{n_k})=0$ for all $k$ and thus $Ty_{n_k}\longrightarrow Ty$ implies that $f(Ty)=0$. Also, we can show that $f_{Ty_{n_k}}(Ty_{n_k})\longrightarrow f(Ty)$. This gives $\lim_{k\longrightarrow\infty} f_{{Ty}_{n_k}}(Ty_{n_k})=\lim_{k\longrightarrow\infty}\|Ty_{n_k}\|=0$, a contradiction to $[T]>0$. Thus,  $x\perp_B y$ implies $Ty\perp_B^{\varepsilon_x}Tx$ for some $\varepsilon_x\in[0,1)$ and for all $y\in X$ and the result follows.
    \end{proof}

 With some minor modifications in the arguments of Theorem~\ref{thmCAOR} and \cite[Theorem 4.2]{CW} we can obtain the following result with strengthens \cite[Theorem 4.2]{CW}.

\begin{theorem}\label{thmsmooth}
    Let $X$ be a  finite-dimensional smooth Banach space. Let $T\in S_{\mathcal{B}(X)}$ be such that $[T]>0$ . Then $T$ is $\varepsilon$-CLAOR with respect to $M_T$ for some $\varepsilon \in[0,1)$.
\end{theorem}

\section{Acknowledgments}
The author would like to thank Jacek Chmieli\'{n}ski (Pedagogical University of Krakow, Polland), Kallol Paul (Jadavpur University, India), and Debmalya Sain (IIIT Raichur, India) for various discussions on orthogonality-related topics in normed linear spaces.
\bigskip

\end{document}